\documentclass{jhrs}

\usepackage{amsfonts, amsmath,amssymb}
\usepackage{hyperref}
\usepackage[all]{xy}
\usepackage[lite]{amsrefs}

\theoremstyle{plain}
\newtheorem{Theorem}{\bfseries Theorem}
\newtheorem{Lemma}{\bfseries Lemma}

\theoremstyle{definition}

\newtheorem{Definition}{\bfseries Definition}
\newtheorem*{notation}{\bfseries Notation}
\newtheorem*{remarkA}{\bfseries Remark A}

\newcommand{\Z}{{\mathbb{Z}}}
\newcommand*{\Homol}{\operatorname{H}}

\begin{document}

\title{Complexifiable characteristic classes}
\author{Alexander D. Rahm}
\email{Alexander.Rahm@nuigalway.ie}
\address{National University of Ireland at Galway
\url{http://www.maths.nuigalway.ie/~rahm/}
}

\thanks{Funded by the Irish Research Council for Science, Engineering and Technology.}
\date{\today}
\begin{abstract} We examine the
topological characteristic cohomology classes of complexified
vector bundles. In particular, all the classes coming from the real vector bundles underlying the complexification are determined.
\\
This article is dedicated to Mark F. Feshbach (1950-2010), for his valuable work on cohomology rings of classifying spaces.
\end{abstract} 

\received{July 14, 2012}   % receive date (for example: 11 October 1999)
\revised{\today}    % receive date
\published{Month Day, Year}  % publish date
\submitted{James D. Stasheff} 
\volumeyear{2013} % Volume Year
\volumenumber{1}  % Volume Number 
\issuenumber{1}   % Issue Number

\startpage{1} 
\maketitle

\setcounter{tocdepth}{3} \setcounter{secnumdepth}{1}

\section{Introduction and statement of the results}

In the theory of characteristic classes (in the sense of Milnor and Stasheff \cite{MilnorStasheff},
 whom we follow in terminology and notation in this article),
 it is well-known how the Chern classes are mapped to even Stiefel-Whitney classes
 when converting complex vector space bundles to real vector space bundles by forgetting the complex structure.
In the other direction, we have the fibre-wise complexification: 
 Given a real vector bundle $F \rightarrow B$ with fibre ${\mathbb{R}}^n$,
 its complexification is the complex vector bundle
 $F^{\mathbb{C}} := F\otimes_{\mathbb{R}} {\mathbb{C}} \rightarrow B$
 obtained by declaring complex multiplication on $F \oplus F$
 in each fibre ${\mathbb{R}}^n \oplus {\mathbb{R}}^n$ by $i(x, y) := (-y, x)$ for the imaginary unit~$i$.
The Pontrjagin classes of a real vector bundle are (up to a sign) constructed as Chern classes of its complexification.
Conversely, which classes of a real vector bundle can be attributed to its complexification?
 These are the \emph{complexifiable} characteristic classes which we determine in this article,
 under the request that they are characteristic classes in the sense of \cite{MilnorStasheff}.

Consider a real vector bundle $F \rightarrow B$ 
and a complex vector bundle $E \rightarrow B $ over the same paracompact Hausdorff base space $B$
(we keep the latter assumption on~$B$ throughout this article).

\begin{Definition} A real vector bundle $F$ is called a
\emph{real generator bundle} of~$E$, if its complexification $F^{\mathbb{C}} $ is isomorphic to~$E$.
In the case that such a bundle~$F$ exists, we call $E$ \emph{real-generated}.
\end{Definition}
Not every complex vector bundle is real-generated; 
as the odd degree Chern classes have the property $c_{2k+1}(\overline{E}) = -c_{2k+1}(E)$ 
on the complex conjugate bundle $\overline{E}$, it is an easy exercise to show that
 no complex vector bundle with some nonzero and non-$2$-torsion odd Chern class can admit a real generator bundle.
This makes it seem possible that the subcategory of real-generated vector bundles could admit information
 additional to its Chern classes, in terms of complexifiable classes of the real generator bundles.
However, we will see that the Chern classes already contain all of the relevant information.

\begin{Definition}
A characteristic class $c$ of real vector bundles is \mbox{\emph{complexifiable}} if for all pairs ($F$, $G$) of real vector bundles with isomorphic complexification $F^{\mathbb{C}} \cong G^{\mathbb{C}}$, the identity $c(F) = c(G)$ holds.
\end{Definition}

We will now give a complete classification of the complexifiable characteristic classes.
Denote by $\Z_2 := \Z/2\Z$ the group with two elements.
\begin{Theorem} Let $c$ be a polynomial in the Stiefel-Whitney classes $w_i$, $i \in {{\mathbb{N}}} \cup \{0\}$. 
Then the following two conditions are equivalent:
\begin{itemize}
 \item[(i)] The class $c$ is an element of the sub-ring
${\mathbb{Z}}_2 [ w_i^2]_{i \in {{\mathbb{N}}} \cup \{0\}}$ of the polynomials in the Stiefel-Whitney classes.
\item[(ii)] The class $c$ is complexifiable.
\end{itemize}
\end{Theorem}
The implication (i)$\Rightarrow$(ii) follows easily from the fact that the square of the
 $n$-th Stiefel-Whitney class of a real vector bundle is the mod-$2$-reduction of the
 $n$-th Chern class of the complexified vector bundle.
The proof of the implication (ii)$\Rightarrow$(i) is prepared with several intermediary steps leading to it.
One ingredient, Lemma~\ref{CartanIdeal}, follows essentially from work of Cartan on fibrations of H-spaces
(at Cartan's time called Hopf spaces).
 But this only allows us to show that complexifiable characteristic classes in cohomology with
 ${\mathbb{Z}}_2$--coefficients are contained in the \emph{ideal} generated by the squares of the Stiefel-Whitney classes.
 To show that they constitute exactly the \emph{subring} generated by the squares of the Stiefel-Whitney classes,
 which is much smaller, we need the technical decomposition of Lemma~\ref{decomposition} that we prove by induction.

By their naturality, characteristic classes are uniquely determined on the universal bundle over the classifying space
 ($B{\mathcal{O}}$ for real vector bundles). 
As the cohomology ring $ \Homol^*(B{\mathcal{O}}, {\mathbb{Z}}_2)$
 is generated by the Stiefel-Whitney classes of the universal bundle,
 all modulo $2$ characteristic classes are polynomials in the Stiefel-Whitney classes,
 and Theorem~1 tells us which of them are complexifiable.

We build on this result to investigate which integral cohomology classes are complexifiable.
To express our result, we use Feshbach's description~\cite{Feshbach} of the cohomology ring of the classifying space
 $B{\mathcal{O}}$ with ${\mathbb{Z}}$--coefficients.
Generators for this ring are known since Thomas~\cite{polynomialAlgebras}, \cite{realGrassmann}, 
and all the relations between its generators are known since Brown~\cite{Brown} and Feshbach~\cite{Feshbach}. 
Consider the Steenrod squaring operation $Sq^1$ and the mod--2--reduction homomorphism 
$$\rho : \Homol^*(B{\mathcal{O}},{\mathbb{Z}})\rightarrow \Homol^*(B{\mathcal{O}},{\mathbb{Z}}_2).$$
As generators for $\Homol^*(B{\mathcal{O}},{\mathbb{Z}})$, 
Feshbach uses Pontrjagin classes and classes $V_I$ with index sets~$I$ that are finite nonempty subsets of
$\left\{ \frac{1}{2} \right\} \cup {\mathbb{N}},$
admitting mod--2--reductions $$\rho (V_I) = Sq^1 \left( \bigcup\limits_{i \hspace{1mm} \in
\hspace{1mm} I} \omega_{2i}\right),$$
where $\omega_i$ is the $i$-th Stiefel Whitney class of the universal bundle over $B{\mathcal{O}}$.
In particular, we have a generator $V_{\{\frac{1}{2}\}}$.
We give the details of Feshbach's description in the appendix.
Our final result now takes the following shape.
\begin{Theorem}
Let $C$ be a polynomial in $V_I^2$, with $I$ arbitrary, $V_{\{\frac{1}{2}\}}$
and the Pontrjagin classes. Then $C$ is complexifiable.
\end{Theorem}
And  conversely, we can say the following.
\begin{Theorem} Let $C$  be a complexifiable integral characteristic class.
Then for any real vector bundle $\xi$,
$C(\xi)$ is completely determined by some Chern classes
$c_k(\xi^{\mathbb{C}})$, $k \in \mathbb{N}$.
\end{Theorem}

\section{Classes in cohomology with ${\mathbb{Z}}_2$--coefficients}

In this section, we shall prove Theorem~1, after developing all the tools we need to do so. 
For this entire section, we only consider cohomology with ${\mathbb{Z}}_2$--coefficients.
We write ${\mathbb{N}}$ for the natural numbers without $0$.
\\
 Let $ F \rightarrow B $ be a real vector bundle over a paracompact Hausdorff base space.
 Let $c$ be a complexifiable polynomial in the Stiefel-Whitney classes $w_i$.
Let ${\mathcal{O}}$ be the direct limit of the orthogonal groups,
$U$ the direct limit of the unitary groups and $EU$ the universal total
space to the classifying space $BU$ for stable complex vector
bundles. Let $B{\mathcal{O}} := EU/{\mathcal{O}}$, via the
inclusion ${\mathcal{O}}\subset U$ induced by the canonical
inclusion ${\mathbb{R}}\subset{\mathbb{C}}$.
Let $\gamma({\mathbb{R}}^\infty)$ be the universal bundle over $B{\mathcal{O}}$, and denote its Stiefel-Whitney classes by
\mbox{$\omega_i :=w_i(\gamma({\mathbb{R}}^\infty)) $}.
Let $\varepsilon$ be the trivial vector bundle.

\begin{Lemma} \label{CartanIdeal}
Let $c$ be a complexifiable class in cohomology with ${\mathbb{Z}}_2$--coefficients. 
\\ Then $ c(\gamma({\mathbb{R}}^\infty)) -c(\varepsilon)\hspace{1mm}$ 
is contained in the ideal
$\langle\omega_i^2\rangle_{i
\hspace{1mm} \in \hspace{1mm} {\mathbb{N}}}
.$
\end{Lemma}

\begin{proof}
We use Cartan's fibration of H-spaces \cite{Cartan}*{p. 17-22} (fibration en espaces de Hopf),
$$ \xymatrix{ U/ {\mathcal{O}} \ar[rr]^f & & B{\mathcal{O}} \ar@{->>}[rr]^p &  & BU.} $$ 
The cohomology ring $\Homol^*(B{\mathcal{O}},{\mathbb{Z}}_2)$ 
is the polynomial algebra ${\mathbb{Z}}_2[\omega_1,\omega_2,...]$ 
with generators the Stiefel-Whitney classes of the universal bundle.
Cartan \cite{Cartan}*{p. 17-22} has shown that $f^*$ maps these
generators $\omega_i$ to the generators $\nu_i :=w_i(f^* \gamma
({\mathbb{R}}^\infty)) $ of the exterior algebra
$$\Homol^*(U/{\mathcal{O}},{\mathbb{Z}}_2) = \bigwedge
({\mathbb{Z}}_2[\nu_1,\nu_2,...]),$$ which is obtained by dividing out the
ideal $\langle \nu_i^2\rangle_{i\hspace{1mm} \in \hspace{1mm} {\mathbb{N}}} $ 
of the polynomial algebra
${\mathbb{Z}}_2[\nu_1,\nu_2,...]$. Hence, exactly the ideal $\langle
\omega_i^2\rangle_{i\hspace{1mm} \in \hspace{1mm}{\mathbb{N}}}$ is mapped to zero. So,
$$
\langle \omega_i^2\rangle_{i \hspace{1mm} \in
\hspace{1mm}{\mathbb{N}}} = \ker f^*.
$$
Composing $f$ with the projection $ p: B{\mathcal{O}} \to BU $, we
obtain a constant map and therefore a trivial bundle $(p\circ f)^*\gamma({\mathbb{C}}^\infty) $. 
This pullback of the complex universal bundle is the complexification of $ f^*\gamma({\mathbb{R}}^\infty)$:
$$ (p\circ f)^*\gamma({\mathbb{C}}^\infty) = f^*p^*EU\times_U
{\mathbb{C}}^\infty =
f^*E{\mathbb{{\mathcal{O}}}}\times_{\mathcal{O}}
{\mathbb{C}}^\infty = f^*(E{\mathcal{O}}\times_{\mathcal{O}}
{\mathbb{R}}^\infty)^{\mathbb{C}}$$
 $$=  f^*\gamma ({\mathbb{R}}^\infty)^{\mathbb{C}} = (
f^*\gamma ({\mathbb{R}}^\infty))\mathbb{^C} .$$  So, $ f^*\gamma
({\mathbb{R}}^\infty)$ admits a trivial complexification, and all
of the complexifiable classes $c$ must treat it like
the trivial bundle $\varepsilon$:
\\
$c(f^* \gamma ({\mathbb{R}}^\infty))=c(\varepsilon)$. A pullback
of the trivial bundle is trivial too, so
$$ 0 = c(f^*\gamma({\mathbb{R}}^\infty)) -c(f^*\varepsilon)
=f^*(c(\gamma ({\mathbb{R}}^\infty)) -c(\varepsilon))$$ by naturality. Whence, $c(\gamma
({\mathbb{R}}^\infty))-c(\varepsilon)$ is an element of the kernel of $f^*$, which we have identified with the ideal $\langle\omega_i^2\rangle_{i
\hspace{1mm} \in \hspace{1mm} {\mathbb{N}}}
.$
\end{proof}

The above lemma allows us to write the characteristic class $c$ under investigation
as a sum over products with squares of Stiefel-Whitney classes,
 $$c(\gamma({\mathbb{R}}^\infty))-c(\varepsilon) = \sum\limits_{j = 1}^m \omega_{i_{j}}^2 \cup r_{j}(\gamma
({\mathbb{R}}^\infty)),$$
 with $r_{j}$ some polynomials in the Stiefel-Whitney classes.
We must inductively identify squares of Stiefel-Whitney classes as factors of the remainders $r_{j}$, 
until we achieve the decomposition claimed in the following lemma.

\vbox{
\begin{notation}
 For indices $j_1,...,j_{s} \in {\mathbb{N}}$ and
$i_{j_1}, ..., i_{(j_1,...,j_s)}  \in {\mathbb{N}}$,
we shall write 
\mbox{$\vec{j}_s := (j_1,...,j_{s})$} and
$I(\vec{j}_s) := \{ i_{\vec{j}_1}, ..., i_{\vec{j}_s} \}$.
We set $\vec{j}_0 := 0$.
\end{notation}
Note that the classes $c(\varepsilon),$
$r_{\vec{j}}(\varepsilon)$ of the trivial bundle $\varepsilon$ 
that we are going to use now, are just coefficients in 
\mbox{$\Homol^0(B{\mathcal{O}},{\mathbb{Z}}_2) \cong \{0,1\}$}.

\begin{Lemma} \label{decomposition}
Any complexifiable characteristic class $c$ admits a decomposition 

\vspace{2mm}

$c(\gamma
({\mathbb{R}}^\infty))-c(\varepsilon) 
= \left(\sum\limits_{j_{k}=1}^{m_{\vec{j}_{k-1}}} \omega_{i_{\vec{j}_k}}^2  r_{\vec{j}_k}(\gamma({\mathbb{R}}^\infty))\right)
\cup \left( \bigcup\limits_{n=1}^{k-1}
\sum\limits_{j_n = 1}^{m_{\vec{j}_{n-1}}} \omega_{i_{\vec{j}_n}}^2 
\right)  
+ \sum\limits_{s = 1}^{k-1} \bigcup\limits_{n=1}^{s}  \sum\limits_{j_{n} = 1}^{m_{\vec{j}_{n-1}}}
\omega_{i_{\vec{j}_{n}}}^2  r_{\vec{j}_{n}}(\varepsilon)$

\vspace{4mm}

 for some $k, m_{\vec{j}_0}, ..., m_{\vec{j}_{k-1}}$
$\in$ ${\mathbb{N}} \cup \{ 0 \}$, some $i_{\vec{j}_1}, ...,
i_{\vec{j}_k}$ $\in$ ${\mathbb{N}}$,

\vspace{2mm}

some $r_{\vec{j}_k}(\gamma ({\mathbb{R}}^\infty))$
$\in$ $\Homol^*(B{\mathcal{O}},{\mathbb{Z}}_2)$,
 and some coefficients $r_{\vec{j}_1}(\varepsilon), ...,
r_{\vec{j}_{k-1}}(\varepsilon)$ $\in$ $ \{ 0, 1 \}$,

\vspace{2mm}

such that the following 
inequality holds: 
$2\sum\limits_{p \hspace{1mm} \in \hspace{1mm} I({\vec{j}_k})} p
> \deg c.$
\end{Lemma}
}

\begin{remarkA}
Once that this lemma is established, we use that the degree must be the
same on both sides 
 in order to deduce that the sum over all terms
containing a factor $ \bigcup\limits_{p \hspace{1mm} \in \hspace{1mm} I(\vec{j}_k)} \omega_p^2$ 
% of too high degree
%$\left(2\sum\limits_{p \hspace{1mm} \in \hspace{1mm} I(\vec{j})} p \right)$ 
exceeding the degree of $c$ via the requested inequality must already be zero.
So in fact, the decomposition is of the form
$$c(\gamma
({\mathbb{R}}^\infty))-c(\varepsilon) 
= \sum\limits_{s = 1}^{k-1} \bigcup\limits_{n=1}^{s}  \sum\limits_{j_{n} = 1}^{m_{\vec{j}_{n-1}}}
\omega_{i_{\vec{j}_{n}}}^2 \cup r_{\vec{j}_{n}}(\varepsilon),$$
meaning that $c(\gamma ({\mathbb{R}}^\infty))$ is a polynomial in
some squares $\omega_p^2$ , \mbox{$p$ $\in$ ${\mathbb{N}} \cup \{ 0\}$}, 
which implies Theorem 1, (ii)$\Rightarrow$(i).
\end{remarkA}

Before giving the proof of Lemma~\ref{decomposition}, 
we shall introduce two notations just to make that proof more readable.
\begin{Definition} An index vector $\vec{j}$
\emph{appears} in a given decomposition of \\
\mbox{$c(\gamma ({\mathbb{R}} ^\infty))-c(\varepsilon)$} 
if both $\left(2\sum\limits_{p \hspace{1mm} \in \hspace{1mm} I(\vec{j})} p \right) \leq \deg c$
\\ 
and this decomposition admits a summand of the form 
$r_{\vec{j}}(\gamma ({\mathbb{R}} ^\infty)) \cup \bigcup\limits_{p \hspace{1mm} \in \hspace{1mm} I(\vec{j})} \omega_p^2$.
\end{Definition} %\begin{Remark} 
Note that the terms $\left( r_{\vec{j}}(\gamma
({\mathbb{R}} ^\infty)) \cup \bigcup\limits_{p \hspace{1mm} \in
\hspace{1mm} I(\vec{j})} \omega_p^2 \right)$ with \hspace{1mm}
$\left(2\sum\limits_{p \hspace{1mm} \in \hspace{1mm} I(\vec{j})} p
> \deg c \right)$ must
vanish in any decomposition of $c(\gamma ({\mathbb{R}}
^\infty))-c(\varepsilon)$. That is why we do not let them contribute
in the last definition.
%\end{Remark}

\begin{Definition} Set $\ell := \min\limits_{\vec{j}
\hspace{1mm} \mathrm{appears}} \max I(\vec{j}). $ Consider an
index vector $\vec{j}$ appearing in a given decomposition of
$c(\gamma ({\mathbb{R}} ^\infty))-c(\varepsilon)$. \\
If $\max I(\vec{j}) = \ell$, then we call $r_{\vec{j}}(\gamma ({\mathbb{R}}
^\infty))- r_{\vec{j}}(\varepsilon)$ a \emph{lower degree remainder}.
\end{Definition}

As seen in Lemma \ref{CartanIdeal}, $c(\gamma ({\mathbb{R}} ^\infty))-c(\varepsilon)$
lies in $ \ker f^* = \langle\omega_i^2\rangle_{i \in
{\mathbb{N}}} $, so there is a decomposition
$$c(\gamma ({\mathbb{R}}^\infty))-c(\varepsilon)
  = \sum\limits_{j_1 = 1}^m \omega_{i_{\vec{j}_1}}^2 \cup r_{\vec{j}_1}(\gamma
({\mathbb{R}}^\infty)),$$ for some $m$ $\in$ ${\mathbb{N}}
\cup \{ 0 \}$, some $i_{\vec{j}_1}$ $\in$ ${\mathbb{N}}$, and some $r_{\vec{j}_1}(\gamma ({\mathbb{R}}^\infty))$ $\in$
$\Homol^*(B{\mathcal{O}},{\mathbb{Z}}_2)$. We will show that there is a
lower degree remainder $r_{\vec{j}_1}(\gamma ({\mathbb{R}}^\infty))
-r_{\vec{j}_1}(\varepsilon)$ in this decomposition that lies in $ \ker
f^*$. Then, that lower degree remainder admits a decomposition as
a linear combination of squares $\omega_{i_{\vec{j}_2}}^2$ with
coefficients $r_{\vec{j}_2}(\gamma ({\mathbb{R}}^\infty))$
in~$\Homol^*(B{\mathcal{O}},{\mathbb{Z}}_2)$, leading to a new decomposition of \mbox{$c(\gamma({\mathbb{R}}^\infty))-c(\varepsilon) $.}
So, inductively, we will replace a lower degree remainder in any
given decomposition of
$c(\gamma({\mathbb{R}}^\infty))-c(\varepsilon) $ by a linear
combination the coefficients of which are remainders with longer index
vectors. That is why after a finite number of these steps, the
index vectors $\vec{j}$ will no longer appear, because the sums
$\left(2\sum\limits_{p \hspace{1mm} \in \hspace{1mm} I(\vec{j})} p\right)$
will exceed the degree of $c$. This is the moment when all lower degree remainders are eliminated and the decomposition described
in Lemma~\ref{decomposition} is achieved.

To carry out this strategy, we first need to introduce the following truncation procedure.

\subsection{Truncated stable invariance}
With Lemma \ref{truncation}, we shall give a sense to ``the truncation of the equation $c(F \oplus G) = c(G)$ at the dimension
$\ell$''. Define the bundles
$$F := pr_1^*f^*\gamma({\mathbb{R}}^\infty)\longrightarrow U/{\mathcal{O}} \times B{\mathcal{O}}$$ 
and
$$G := pr_2^*\gamma({\mathbb{R}}^\infty)\longrightarrow U/{\mathcal{O}} \times B{\mathcal{O}},$$
where $pr_i$ is the projection on the $i$-th factor of the
base space \mbox{$U/{\mathcal{O}} \times B{\mathcal{O}}$.} 
Let $\ell \in  {\mathbb{N}}$. Consider the map
$$(id, emb_l): (U/ {\mathcal{O}}\times B{\mathcal{O}}_\ell)
\hookrightarrow (U/ {\mathcal{O}}\times B{\mathcal{O}})$$ where
$emb_l:  B{\mathcal{O}}_\ell \hookrightarrow B{\mathcal{O}}$ is the natural embedding into the direct limit.  
Then the bundle $G_l := (id, emb_l)^*G$ 
admits Stiefel-Whitney classes that are in bijective correspondence with those of the $\ell$-dimensional universal bundle
$\gamma_l({\mathbb{R}}^\infty)\rightarrow B{\mathcal{O}}_\ell$.

To be precise, $G_l \cong
{pr_{B{\mathcal{O}}_\ell}}^*\gamma_l({\mathbb{R}}^\infty)$ and  the
situation is $$\xymatrix{ \gamma_l({\mathbb{R}}^\infty) \ar[d]
& G_l \cong {pr_{B{\mathcal{O}}_\ell}}^*\gamma_l({\mathbb{R}}^\infty)
\ar[d] & G:=pr_2^*\gamma({\mathbb{R}}^\infty) \ar[d] &
\gamma({\mathbb{R}}^\infty) \ar[d]
\\ B{\mathcal{O}}_\ell & (U/ {\mathcal{O}}\times
B{\mathcal{O}}_\ell) \ar[l]_{pr_{B{\mathcal{O}}_\ell}}
\ar@{^{(}->}[r]^{(id,emb_l)} & (U/ {\mathcal{O}}\times
B{\mathcal{O}}) \ar[r]^{pr_2} & B{\mathcal{O}}. }$$\hfill

Especially, $w_p(G_l)$ vanishes for $p > \ell$.
Compare the latter statements with~\cite{MilnorStasheff}.

\begin{Lemma} \label{truncation}
Under the above assumptions, the following equation holds:
$$\sum\limits_{\vec{j} \hspace{1mm}
{\mathrm{appears}}}^{ \max I(\vec{j}) \hspace{1mm}\leq
\hspace{1mm}\ell } r_{\vec{j}}(F \oplus G_l)
\bigcup\limits_{p\hspace{1mm} \in \hspace{1mm} I(\vec{j})}
w_p^2(G_l)=\sum\limits_{\vec{j} \hspace{1mm} {\mathrm{appears}}}^{
\max I(\vec{j}) \hspace{1mm}\leq \hspace{1mm}\ell } r_{\vec{j}}(G_l)
\bigcup\limits_{p\hspace{1mm} \in \hspace{1mm} I(\vec{j})}
w_p^2(G_l).$$
We will call it \emph{the equation $c(F \oplus G) = c(G)$ truncated at dimension~$\ell$}.
\end{Lemma}

\begin{proof}
The bundle $F$ inherits from $f^*\gamma({\mathbb{R}}^\infty)$ the
property of admitting a trivial complexification.
 As $c$ is complexifiable, we have $c(F \oplus G) = c(G).$
Applying the induced cohomology map $(id, emb_l)^*$ to this equation, we obtain
$$c(id^*F \oplus emb_l^*G) = c(emb_l^*G)$$
and hence
$$ c(F \oplus G_l) = c(G_l).$$
By the universality of $\gamma({\mathbb{R}}^\infty)$, and the
naturality of all characteristic classes with respect to the classifying
maps of $G_l$ and $F \oplus G_l$, any given decomposition
 $$c(\gamma
({\mathbb{R}}^\infty)) - c(\varepsilon) = \sum\limits_{\vec{j}}
r_{\vec{j}}(\gamma ({\mathbb{R}}^\infty))
\bigcup\limits_{p\hspace{1mm} \in \hspace{1mm} I(\vec{j})}
\omega_p^2$$  gives analogous decompositions
 $$c(G_l) - c(\varepsilon) =
\sum\limits_{\vec{j}} r_{\vec{j}}(G_l)
\bigcup\limits_{p\hspace{1mm} \in \hspace{1mm} I(\vec{j})}
w_p^2(G_l)$$  and $$c( F \oplus G_l) - c(\varepsilon) =
\sum\limits_{\vec{j} } r_{\vec{j}}(F \oplus G_l)
\bigcup\limits_{p\hspace{1mm} \in \hspace{1mm} I(\vec{j})}
w_p^2(F \oplus G_l).$$ By Theorem 1, (i)$\Rightarrow$(ii) the square $w_p^2$ is complexifiable
and hence invariant under adding the bundle $F$ of trivial complexification :
$$w_p^2(F \oplus G_l)= w_p^2(G_l).$$
Thus, the equation $c(F \oplus G_l) = c(G_l)$ can be rewritten using that all summands containing a factor
 $w_p(G_l)$ with $p > \ell$
vanish:
$$ \sum\limits_{\vec{j} }^{ \max
I(\vec{j}) \hspace{1mm}\leq \hspace{1mm}\ell } r_{\vec{j}}(F \oplus
G_l) \bigcup\limits_{p\hspace{1mm} \in \hspace{1mm}
I(\vec{j})} w_p^2(G_l)=\sum\limits_{\vec{j}}^{ \max I(\vec{j})
\hspace{1mm}\leq \hspace{1mm}\ell } r_{\vec{j}}(G_l)
\bigcup\limits_{p\hspace{1mm} \in \hspace{1mm} I(\vec{j})}
w_p^2(G_l).$$
In order not to exceed the degree of $c$, also all terms
with $2\sum\limits_{p \hspace{1mm} \in \hspace{1mm}
I(\vec{j})} p
> \deg c$ must vanish:
$$ \sum\limits_{\vec{j} \hspace{1mm}
{\mathrm{appears}}}^{ \max I(\vec{j}) \hspace{1mm}\leq
\hspace{1mm}\ell } r_{\vec{j}}(F \oplus G_l)
\bigcup\limits_{p\hspace{1mm} \in \hspace{1mm} I(\vec{j})}
w_p^2(G_l)=\sum\limits_{\vec{j} \hspace{1mm} {\mathrm{appears}}}^{
\max I(\vec{j}) \hspace{1mm}\leq \hspace{1mm}\ell } r_{\vec{j}}(G_l)
\bigcup\limits_{p\hspace{1mm} \in \hspace{1mm} I(\vec{j})}
w_p^2(G_l).$$
So, this last equation is the
equation $c(F \oplus G) = c(G)$ truncated at the dimension~$\ell$.
\end{proof}

\textit{Proof of Lemma \rm{\ref{decomposition}}}.
We carry out the proof by \textit{induction over the index vector identifying a lower degree remainder}.
${}$ \\
\textbf{Base case}. Lemma \ref{CartanIdeal} implies $c(\gamma
({\mathbb{R}}^\infty))-c(\varepsilon) $
 $ = \sum\limits_{j_1 = 1}^m \omega_{i_{\vec{j}_1}}^2 \cup r_{\vec{j}_1}(\gamma
({\mathbb{R}}^\infty))$, \\
 with $r_{\vec{j}_1}$ some polynomials in the Stiefel-Whitney classes.
\\
Rename $i_1, ..., i_m$ such that $i_1 < i_2 < ... < i_m$.
\\
We truncate the equation $c(F \oplus G) = c(G)$ at the dimension $i_1 $, and obtain
$$\sum\limits_{\vec{j}_1 \hspace{1mm} {\mathrm{appears}}}^{ i_{j_1}
\hspace{1mm}\leq \hspace{1mm} i_1 } r_{\vec{j}_1}(F \oplus G_{i_1}) \cup
w_{i_{\vec{j}_1}}^2(G_{i_1}) 
= \sum\limits_{\vec{j}_1 \hspace{1mm}{\mathrm{appears}}}^{ i_{j_1} \hspace{1mm}\leq \hspace{1mm} i_1 }
r_{\vec{j}_1}(G_{i_1}) \cup w_{i_{\vec{j}_1}}^2(G_{i_1}).$$
As $i_1 < i_2 < ... < i_m$,
this is just $ r_{1}(F \oplus G_{i_1}) \cup
w_{i_{1}}^2(G_{i_1})$ $ =
 r_{1}(G_{i_1}) \cup
w_{i_{1}}^2(G_{i_1})$.
\\
\\
Injectivity of the multiplication map $\cup w_{i_{1}}^2(G_{i_1})$
in $\Homol^*(U/ {\mathcal{O}}\times B{\mathcal{O}}_{i_1},
{\mathbb{Z}}_2)$ then holds
\\
 $ r_{1}(F \oplus G_{i_1}) = r_{1}(G_{i_1}) $.
Then we pull this back with $$(id \times const) : U/ {\mathcal{O}}
\rightarrow (U/ {\mathcal{O}}\times B{\mathcal{O}}_{i_1}),$$
(where the map $const$ takes just one, arbitrary, value), to obtain
$$r_1(f^*\gamma ({\mathbb{R}}^\infty) \oplus \varepsilon)
 = r_1(\varepsilon).$$
 Due to the Whitney sum formula, the Stiefel-Whitney classes in which $r_1$ is a polynomial are stable under adding a trivial bundle; 
 and the above left hand term equals
$r_1(f^*\gamma ({\mathbb{R}}^\infty)).$
\\
 Using naturality of characteristic classes with respect to pullbacks,
 this shows that
\\
\mbox{$r_1(\gamma({\mathbb{R}}^\infty)) - r_1(\varepsilon)$} lies in
$\ker f^*$. So we can replace it with a linear (over the field with 2 elements) combination of strictly quadratic terms, providing a new decomposition,
$$c(\gamma ({\mathbb{R}}^\infty))-c(\varepsilon)  = \omega_{i_1}^2  \sum\limits_{j_2 = 1}^{m_1} \omega_{i_{(1,j_2)}}^2  r_{(1,j_1)}(\gamma
({\mathbb{R}}^\infty)) + \omega_{i_1}^2  r_1(\varepsilon) + \sum\limits_{j_1 = 2}^m \omega_{i_{j_1}}^2  r_{j_1}(\gamma
({\mathbb{R}}^\infty)).$$
\textbf{Induction hypothesis}.
Consider a given decomposition $$c(\gamma({\mathbb{R}}^\infty))-c(\varepsilon)  
= \left(\sum\limits_{\vec{j}_k}
r_{\vec{j}_k}(\gamma ({\mathbb{R}}^\infty))
\bigcup\limits_{p\hspace{1mm} \in \hspace{1mm} I(\vec{j}_k)}
\omega_p^2\right)
 +\sum\limits_{s = 1}^{k-1} \bigcup\limits_{n=1}^{s}  \sum\limits_{j_{n} = 1}^{m_{\vec{j}_{n-1}}}
\omega_{i_{\vec{j}_{n}}}^2 \cup r_{\vec{j}_{n}}(\varepsilon).$$
 \textbf{Inductive claim}. The decomposition of the induction hypothesis admits a lower degree remainder that lies in $\ker f^*$.
We show this in the inductive step.
\\
\\
\textbf{Inductive step}. We truncate the equation $c(F \oplus G) = c(G)$
at the dimension $$\ell := \min\limits_{\vec{j} \hspace{1mm}
{\mathrm{appears}}} \max I(\vec{j}).$$ Then the remaining terms of
$c(G_l) - c(\varepsilon)$ do all have the common factor
$w_l^2(G_l)$. This is not a zero divisor in $\Homol^*(U/{\mathcal{O}}\times B{\mathcal{O}}_{\ell}, {\mathbb{Z}}_2)$ 
and furthermore its multiplication map $ \cup w_l^2(G_l)$ is injective.
Now, in $ c(F \oplus G_l) = c(G_l)$, this injectivity implies $$\sum\limits_{\vec{j} \hspace{1mm}
{\mathrm{appears}}}^{ \max I(\vec{j}) \hspace{1mm}\leq
\hspace{1mm}\ell } r_{\vec{j}}(F \oplus G_l)
\bigcup\limits_{p\hspace{1mm} \in \hspace{1mm} I(\vec{j})
\setminus \{ \ell \}} w_p^2(G_l)=\sum\limits_{\vec{j} \hspace{1mm}
{\mathrm{appears}}}^{ \max I(\vec{j}) \hspace{1mm}\leq
\hspace{1mm}\ell } r_{\vec{j}}(G_l) \bigcup\limits_{p\hspace{1mm}
\in \hspace{1mm} I(\vec{j}) \setminus \{ \ell \}} w_p^2(G_l).$$

\vspace{4mm}

$\diamondsuit$ If there is just one lower degree remainder
$r_{\vec{j}}(\gamma ({\mathbb{R}} ^\infty))-
r_{\vec{j}}(\varepsilon)$, then we use the injectivity of the
multiplication map 
\begin{center}
$ \cup \left( \bigcup\limits_{p\hspace{1mm} \in
\hspace{1mm} I(\vec{j}) \setminus \{ \ell \}} w_p^2(G_l) \right)$
on $\Homol^*(U/ {\mathcal{O}}\times B{\mathcal{O}}_{\ell}, {\mathbb{Z}}_2)$
\end{center}
to obtain $ r_{\vec{j}}(F \oplus G_l) =
r_{\vec{j}}(G_l)$.
Then we pull this back with
$$(id \times const) : U/ {\mathcal{O}} \rightarrow (U/
{\mathcal{O}}\times B{\mathcal{O}}_{\ell})$$
to obtain
$r_{\vec{j}}(f^*\gamma ({\mathbb{R}} ^\infty) \oplus
\varepsilon)= r_{\vec{j}}(\varepsilon)$. 
Using naturality, we see now that the
lower degree remainder $r_{\vec{j}}(\gamma ({\mathbb{R}}
^\infty))- r_{\vec{j}}(\varepsilon)$ lies in $\ker f^*$.

%\vspace{6mm}

\vspace{6mm}

 $\diamondsuit$ Otherwise, we truncate the remaining equation
again at the dimension $$\ell' := {\min\limits_{\vec{j} \hspace{1mm}
{\mathrm{appears}}}^{\max I(\vec{j}) = \ell}} \max (I(\vec{j})
\setminus \{ \ell\}),$$
so as to obtain 
$$\sum\limits_{\vec{j} \hspace{1mm} {\mathrm{appears}}}^{ \max
(I(\vec{j}) \setminus \{ \ell\} ) \hspace{1mm}\leq \hspace{1mm}\ell' }
r_{\vec{j}}(F \oplus G_{\ell'}) \bigcup\limits_{p\hspace{1mm}
\in \hspace{1mm} (I(\vec{j})\setminus \{ \ell\} )}
w_p^2(G_{\ell'})$$ 
$$= \sum\limits_{\vec{j} \hspace{1mm}
{\mathrm{appears}}}^{ \max (I(\vec{j}) \setminus \{ \ell\} )
\hspace{1mm}\leq \hspace{1mm}\ell' } r_{\vec{j}}(G_{\ell'})
\bigcup\limits_{p\hspace{1mm} \in \hspace{1mm}
(I(\vec{j})\setminus \{ \ell\} )} w_p^2(G_{\ell'}).$$ \normalsize
Now we proceed analogously with the choice marked with the ``$\diamondsuit$'' signs
and, after finitely many steps, find a lower degree remainder in $\ker f^*$.
This lower degree remainder can be replaced by a linear
combination of squares, holding a new decomposition of $c(\gamma
({\mathbb{R}}^\infty))-c(\varepsilon) $.
This completes the induction. \hfill $\Box$

\begin{proof}[Proof of Theorem 1, (ii)$\Rightarrow$(i)]
Let $c$ be a complexifiable characteristic class. By Remark A and the universality of $\gamma ({\mathbb{R}}^\infty)$,
 the decomposition of Lemma~\ref{decomposition} yields the decomposition
$$c =  c(\varepsilon) +
\sum\limits_{s = 1}^{k-1} \bigcup\limits_{n=1}^{s}  \sum\limits_{j_{n} = 1}^{m_{\vec{j}_{n-1}}}
w_{i_{\vec{j}_{n}}}^2 \cup r_{\vec{j}_{n}}(\varepsilon).$$
 As $c(\varepsilon),  r_{\vec{j}_1}(\varepsilon), ..., r_{\vec{j}_{k-1}}
 (\varepsilon)$ are elements of $\{0, 1 = w_0 = w_0^2\},$
the class~$c$ is in the sub-ring \mbox{${\mathbb{Z}}_2 [ w_i^2]_{i \hspace{1mm}
\in \hspace{1mm}{{\mathbb{N}}} \cup \{0\}}$} of the polynomial
ring of Stiefel-Whitney classes.
\end{proof}
This completes the proof of Theorem 1.

\section{Classes in cohomology with integral coefficients}
We will build on our results obtained for
${\mathbb{Z}}_2$--coefficients and use the mod--2--reduction
homomorphism $$\rho : \Homol^*(-,{\mathbb{Z}})\rightarrow
\Homol^*(-,{\mathbb{Z}}_2)$$ to prove the theorems with
${\mathbb{Z}}$--coefficients stated in the introduction. Define the element $V_I \in \Homol^*(B{\mathcal{O}},{\mathbb{Z}})$ as in the appendix,
 and let $v_I$ be the characteristic class that is $V_I$ on the universal bundle.

\vspace{7mm}

\begin{Lemma} \label{polynomial}
For any real bundle $\xi$, the mod--$2$--reduced class $\rho(v_I^2(\xi))$ equals
$$\left( \sum\limits_{i \hspace{1mm} \in \hspace{1mm} I \cap \{
\frac{1}{2} \}} w_1^2 \cup \bigcup\limits_{j
\hspace{1mm} \in \hspace{1mm} I \setminus \{ i \} }
w_{4j}+\sum\limits_{i \hspace{1mm} \in \hspace{1mm} I \setminus \{
\frac{1}{2} \}} (w_{4i+2}+w_{2}\cup
w_{4i}) \cup \bigcup\limits_{j \hspace{1mm} \in
\hspace{1mm} I \setminus \{ i \} } w_{4j}\right)(\xi \oplus \xi). $$
\end{Lemma}

\begin{proof} By Feshbach's description (in the appendix), the mod--2--reduction is
$$\rho\left(v_I^2(\xi)\right) = \left(Sq^1 \left(\bigcup\limits_{i \hspace{1mm} \in \hspace{1mm} I}
w_{2i}(\xi)\right)\right)^2.$$
We expand this expression until it is a polynomial in the Stiefel-Whitney classes.
Then we rearrange the expression using the Whitney sum formula and the symmetry of the terms.
\end{proof}

\begin{proof}[Proof of Theorem 2]
 For $v_{\{\frac{1}{2}\}}$ and the Pontrjagin classes, the result is obvious. 
Now let $F\rightarrow B$, $G\rightarrow B $ be real bundles with $F^{\mathbb{C}} \cong G^{\mathbb{C}} $.
Forgetting the complex structure, this is $ F \oplus F \cong G \oplus G $. 
By naturality of the Stiefel-Whitney classes,
 for any finite nonempty index set $I$ $\subset$ $(\{ \frac{1}{2}
\} \cup {\mathbb{N}})$, the polynomial
given in Lemma~\ref{polynomial} 
is the same for the arguments $(F \oplus F)$ and $(G \oplus G)$. Applying Lemma~\ref{polynomial}, this
means that $\rho(v_I^2(F)) = \rho(v_I^2(G)) $.
As $V_I^2$ is in the torsion of
$\Homol^*(B{\mathcal{O}},{\mathbb{Z}})$, restricted on which $\rho$ is
injective \cite{Feshbach}{p. 513}, this proves the theorem:
$v_I^2(F) = v_I^2(G) $.
\end{proof}

\begin{proof}[Proof of Theorem 3]
Feshbach \cite{Feshbach}{p. 513} shows that
\begin{center}
$\Homol^*(B{\mathcal{O}},{\mathbb{Z}}) = {\mathbb{Z}}[ \pi_i]_{i
\hspace{1mm} \in \hspace{1mm}{{\mathbb{N}}}} \oplus $
\{2--torsion\},
\end{center}
where $\pi_i$ is the $i$-th Pontrjagin class of the universal bundle.
Then $C = P( p_i) + T$ with $P$ a polynomial in the Pontrjagin classes $p_i$ and $T$ some
2-torsion class. 
So for every real bundle~$\xi$,
 $$\rho(C)(\xi) = P\left(\rho\left( p_i(\xi)\right)\right) +  \rho(T)(\xi).$$
By definition of the Pontrjagin classes,
 $p_i(\xi) = (-1)^i c_{2i}(\xi^{\mathbb{C}})$ ; and
using the reduction $\rho\left( c_{2i}(\xi^{\mathbb{C}})\right)= w_{4i}(\xi \oplus \xi)$
 from Chern classes to Stiefel-Whitney classes,
 further the Whitney sum formula and the symmetry of the summands,
 we deduce $$\rho(C)(\xi) = P(w_{2i}^2(\xi)) +\rho(T)(\xi).$$
It follows from Theorem 1 that the mod-$2$-reduction $\rho(C)(\xi)$
is a polynomial in the squares of Stiefel-Whitney classes;
and hence also $\rho(T)(\xi)$ is a polynomial $ Q( w_{j}^2(\xi))$ in the squares of Stiefel-Whitney classes.
 As according to \cite{Feshbach}{p. 513},
 $\rho$ is  injective on the torsion elements, there is an inverse for the restricted map  ${\rho|_{\{2\mathrm{-torsion}\}}}$,
 lifting $\rho(T)$ back to $T$. So, from
 $$T(\xi) = {\rho|_{\{2\mathrm{-torsion}\}}}^{-1}\left(Q( w_{j}^2(\xi)) \right),$$
we obtain 
$$C(\xi) = P\left( (-1)^i
c_{2i}(\xi^{\mathbb{C}})\right) +
{\rho|_{\{2\mathrm{-torsion}\}}}^{-1}\left(Q\left( \rho(c_{j}(\xi^{\mathbb{C}}))\right)\right).$$ 
\end{proof}

The author would like to thank Graham Ellis and Thomas Schick for support and encouragement,
the latter also for posing the questions treated in this article and giving advice on them.

\bigskip
\bigskip
\bigskip
\bigskip
\bigskip
\bigskip
\bigskip

\section*{Appendix. The cohomology ring of $B{\mathcal{O}}$ with ${\mathbb{Z}}$--coefficients}

The cohomology ring of $B{\mathcal{O}}$ with
${\mathbb{Z}}$--coefficients is known since Thomas \cite{polynomialAlgebras}, \cite{realGrassmann} 
and with all relations between its generators since Brown \cite{Brown} and Feshbach \cite{Feshbach}. 
It can be derived as follows.
Define the set of generators of
$\Homol^*(B{\mathcal{O}}_n,{\mathbb{Z}})$ as \mbox{in \cite{Feshbach}*{definition 1}:}
\\
It consists of the Pontrjagin classes $p_i$ of the universal
bundle over $B{\mathcal{O}}_n$, and classes $V_I$ with $I$ ranging
over all finite nonempty subsets of
$$\left\{ \frac{1}{2} \right\} \cup \left\{ k \hspace{1mm} \in \hspace{1mm}
{\mathbb{N}} \hspace{1mm} \left|\hspace{1mm} 0<k<\frac{n+1}{2} \right. \right\}$$
with the proviso that $I$ does not contain both $\frac{1}{2}$ and
$\frac{n}{2}$, for $n>1$.
\\
According to \cite{Feshbach}*{theorem 2},
$\Homol^*(B{\mathcal{O}}_n,{\mathbb{Z}})$ is for all $n \leq \infty$
isomorphic to the polynomial ring over ${\mathbb{Z}}$ generated by
the above specified elements modulo the ideal generated by the
following six types of relations.
\\
In all relations except the first, the cardinality of $I$ is less
than or equal to that of $J$ and greater than one. 
On the index sets $I$ and $J$, we perform set-theoretic unions ($\cup$), intersections ($\cap$) and differences ($\setminus$).
By convention, $p_{\frac{1}{2}}$ where it occurs means $V_{\left\{\frac{1}{2} \right\}}$.
Also, if $\left\{ \frac{n}{2}, \frac{1}{2} \right\} \subset
\hspace{1mm} I \cup J$, then $V_{I \cup J}$ shall mean $V_{\left\{
\frac{n}{2} \right\}} V_{(I \cup J) \setminus \left\{ \frac{n}{2},
\frac{1}{2} \right\}}$.
As Feshbach remarks, most of the restrictions on $I$ and $J$ are to avoid repeating relations. 

\vspace{2mm}

 $1) \hspace{3mm} 2 V_I = 0$.

\vspace{2mm}

$2) \hspace{3mm} V_I V_J + V_{I \cup J}V_{I \cap J} +V_{I
\setminus J} V_{J \setminus I} \prod\limits_{i \hspace{1mm}
\in \hspace{1mm} I \cap J} p_i = 0$ \hspace{3mm} (for $ I
\cap J \neq \emptyset , \hspace{2mm} I \nsubseteq J)$.

\vspace{2mm}

$3) \hspace{3mm} V_I V_J + \sum\limits_{i \hspace{1mm} \in
\hspace{1mm} I} V_{\{i\}} V_{(J \setminus I) \cup \{i\}}
\prod\limits_{j \hspace{1mm} \in \hspace{1mm} I \setminus
\{i\}} p_j = 0$ \hspace{3mm} (for $ I \subset J $).

\vspace{2mm}

$4) \hspace{3mm} V_I V_J + \sum\limits_{i \hspace{1mm} \in
\hspace{1mm} I} V_{\{i\}} V_{(I \cup J) \setminus \{i\}} = 0$
\hspace{2mm} (for $ I \cap J = \emptyset$; if $I$ and $J$ have the

same cardinality, then the smallest element of $I$ is to be less than that of $J$).

\vspace{2mm}

$5) \hspace{3mm} \sum\limits_{i \hspace{1mm} \in \hspace{1mm}
I} V_{\{i\}} V_{I \setminus \{i\}} = 0$. \vspace{2mm}

$6) \hspace{3mm} V_{\{ \frac{1}{2} \}} p_\frac{n}{2} + V_{\{
\frac{n}{2} \}}^2 = 0$, if $n$ is even.

\vspace{8mm}

Then $\rho (V_I) = Sq^1 ( \bigcup\limits_{i \hspace{1mm} \in
\hspace{1mm} I} w_{2i})$.

\end{document}